\documentclass[oneside,english]{amsart}
\usepackage[T1]{fontenc}
\usepackage[latin9]{inputenc}
\usepackage{geometry}
\usepackage{amsthm}
\usepackage{amssymb}
\usepackage{amsmath}

\makeatletter
\numberwithin{equation}{section}
\numberwithin{figure}{section}
\theoremstyle{plain}
\newtheorem{theorem}{\protect\theoremname}
\theoremstyle{plain}
\newtheorem{entry}{Entry}
\newtheorem{proposition}[theorem]{Proposition}

\newcommand{\A}{\mathbf{A}}
\newcommand{\M}{\mathbf{M}}
\makeatother

\usepackage{babel}

\providecommand{\theoremname}{Theorem}

\begin{document}

\title{Dubious Identities:
A Visit to the Borwein Zoo}
\markright{Dubious Identities}
\author{Zachary P. Bradshaw and Christophe Vignat}
\address{Department of Mathematics, Tulane University, New Orleans LA 70118\\
zbradshaw@tulane.edu}
\address{Department of Physics, Universit\'{e} Paris Saclay,
L.S.S, CentraleSup\'{e}lec, Orsay, 91190, France\\
Department of Mathematics, Tulane University, New Orleans LA 70118\\
cvignat@tulane.edu, 
christophe.vignat@universite-paris-saclay.fr}

\maketitle

\begin{abstract}
We contribute to the zoo of dubious identities established by J.M. and P.B. Borwein in their 1992 paper, ``Strange Series and High Precision Fraud'' with five new entries, each of a different variety than the last. Some of these identities are again a high precision fraud and picking out the true from the bogus can be a challenging task with many unexpected twists along the way. 
\end{abstract}

\section*{Introduction.}

The article \cite{Borwein} by J.M. and P.B. Borwein features twelve examples
of dubious identities, some of them being true. It has been a source
of inspiration for many mathematicians since it appeared in 1992. 
We humbly propose to extend this list to areas that were not originally touched upon with five new items of our own that are in the spirit of \cite{Borwein}, hoping that this work will spark others to contribute to this zoo of rarities. Among the next identities, in the words of the Borweins, ``the reader is invited to separate the true from the bogus.'' 

\begin{entry}\label{Student sums}
For any real  $\lambda>\frac{1}{2}$ and with $B\left(a,b\right)=\frac{\Gamma\left(a\right)\Gamma\left(b\right)}{\Gamma\left(a+b\right)}$ the Euler beta function,
\[
\frac{1}{10^{5}}\sum_{n\in\mathbb{Z}}\left(1+\frac{n^{2}}{10^{10}}\right)^{-\lambda}=B\left(\frac{1}{2},\lambda-\frac{1}{2}\right).
\]
\end{entry}

\begin{entry}
\label{Infinite matrix}
With $\zeta\left(s\right)$ the Riemann zeta function,
\[
\prod_{n\ge1}\left[\begin{array}{ccc}
\frac{n}{2\left(2n+1\right)} & \frac{-3}{2n\left(2n+1\right)} & \frac{3}{2n^{3}}\\
0 & \frac{n}{2\left(2n+1\right)} & \frac{3}{2n}\\
0 & 0 & 1
\end{array}\right]=\left[\begin{array}{ccc}
0 & 0 & \zeta\left(4\right)\\
0 & 0 & \zeta\left(2\right)\\
0 & 0 & 1
\end{array}\right]
\]
and
\[
\prod_{n\ge1}\left[\begin{array}{cccc}
\frac{n}{2\left(2n+1\right)} & \frac{-3}{2n\left(2n+1\right)} & 0 & \frac{3}{2n^{5}}\\
0 & \frac{n}{2\left(2n+1\right)} & \frac{-3}{2n\left(2n+1\right)} & \frac{3}{2n^{3}}\\
0 & 0 & \frac{n}{2\left(2n+1\right)} & \frac{3}{2n}\\
0 & 0 & 0 & 1
\end{array}\right]=\left[\begin{array}{cccc}
0 & 0 & 0 & \zeta\left(6\right)\\
0 & 0 & 0 & \zeta\left(4\right)\\
0 & 0 & 0 & \zeta\left(2\right)\\
0 & 0 & 0 & 1
\end{array}\right].
\]
\end{entry}

\begin{entry}
\label{discrete normal}
It holds that
\[
\sum_{n\in\mathbb{Z}}n^{2}\exp\left(-\frac{n^{2}}{2}\right)=\sum_{n\in\mathbb{Z}}\exp\left(-\frac{n^{2}}{2}\right)
\]
and
\[
\sum_{n\in\mathbb{Z}}n^{4}\exp\left(-\frac{n^{2}}{2}\right)=3\sum_{n\in\mathbb{Z}}\exp\left(-\frac{n^{2}}{2}\right).
\]

\end{entry}

\begin{entry} 
\label{coloring}
For $n\ge1,$
\[
\sum_{j_{1}+\cdots+nj_{n}=n}\frac{n^{j_{1}+\cdots+j_{n}}}{\prod_{k=1}^{n}k^{j_{k}}j_{k}!}=\binom{2n-1}{n}.
\]
\end{entry}

\begin{entry}
\label{digits}
Let $a(n)$ and $b\left(n\right)$ denote the number
of even and odd digits, respectively, in the decimal expansion of $n$. Define the
sequence $c(n):=10^{5}a(n)-b(n)/10^{5}$. Then
\[
\sum_{n=0}^{\infty}\frac{c(n)}{10^{n}}=\frac{11111111111}{110000}.
\]
\end{entry}

The reader may wish to skip this paragraph or pause before reading on, as it 
contains spoilers. Entry \ref{Student sums} is a variation of the twelfth sum in Borwein's paper. It is only an approximation, with an error bounded by $10^{-136440}$ uniformly over $\lambda > \frac{1}{2}$. The first identity in Entry \ref{Infinite matrix} is true while the second one is unexpectedly false; two correct and equivalent versions are given in \eqref{eq:zeta(6) correct version} and \eqref{alternate version} of Section \ref{sec:zeta}. Both are extensions of results originally derived by B. Gosper \cite{Gosper} about the representation of series as infinite products of matrices. 
Both identities in Entry \ref{discrete normal} are false, although numerically almost satisfied; they appeal to the theory of Jacobi theta functions and can be interpreted in a probabilistic setup that was developed in \cite{RomikTheta} and extended in \cite{Romik}. Entry \ref{coloring} is correct and directly related to the theory of colorings with respect to a group symmetry. A generalization of this identity with a discussion of the relation to coloring theory is contained in Section~\ref{sec:partitionsum}. The identity in Entry \ref{digits} is only an approximation and has an error on the order of $10^{-105}$. Interestingly, this approximation can be extended in a natural way which is discussed in Section~\ref{sec:hpfraud}. 

\section{A Variation of Borwein's Twelfth Sum.}

Before studying Entry \ref{Student sums}, let us discuss one of the most fascinating
identities in \cite{Borwein}, namely Sum 12:

\begin{equation}
\frac{1}{10^{5}}\sum_{n\in\mathbb{Z}}\exp\left(-\frac{n^{2}}{10^{10}}\right)=\sqrt{\pi}.\label{eq:normal sum}
\end{equation}
This is only an approximation, but an excellent one; the authors use
the modularity property of the Jacobi $\theta_{3}$ function to compute
the impressive bound 
\[
\bigg\vert\sqrt{\pi}-\frac{1}{10^{5}}\sum_{n\in\mathbb{Z}}\exp\left(-\frac{n^{2}}{10^{10}}\right)\bigg\vert\le10^{-4.2\times10^{10}}.
\]
However, modularity is not the reason why this approximation is so
accurate. The real reason is that the terms in the series can
be interpreted as the samples at integer values of the argument $z$
of a continuous normal distribution $f\left(z\right)=\frac{1}{\sigma\sqrt{2\pi}}\exp\left(-\frac{z^{2}}{2\sigma^{2}}\right)$ with zero mean and a very large variance, namely $\sigma^{2}=\frac{1}{2}10^{10}$. This means that the series (\ref{eq:normal sum}) is very close to
the Riemann sum, computed with a very fine subdivision, 
of a standard normal continuous
distribution, hence very close to the theoretical value 
\[
\int_{-\infty}^{\infty}e^{-x^{2}}dx=\sqrt{\pi}.
\]
This accuracy result is thus a scaling effect, and we observe that it can be
obtained for any probability density function $f\left(x\right)$ since,
for an arbitrary scale factor $a>0$, 
\[
\int_{\mathbb{R}}f\left(x\right)dx=1\implies\frac{1}{a}\int_{\mathbb{R}}f\left(\frac{x}{a}\right)dx=1.
\]
Now choose a very large value of the parameter $a$ and approximate
the integral with a Riemann sum: 
\[
\frac{1}{a}\sum_{k\ge0}^{\infty}f\left(\frac{k}{a}\right)\simeq\frac{1}{a}\int_{0}^{\infty}f\left(\frac{x}{a}\right)dx=1.
\]
In the context of Entry~\ref{Student sums}, we consider the family of Student t-distributions 
\[
f_{\lambda}\left(x\right)=\frac{1}{B\left(\frac{1}{2},\lambda-\frac{1}{2}\right)}\left(1+x^{2}\right)^{-\lambda},\,\,x\in\mathbb{R},\,\,\lambda > \frac{1}{2}.
\]

\begin{proposition}
We have the approximation, for any $\lambda > \frac{1}{2},$
\begin{equation}
\frac{1}{10^{5}}\sum_{n\in\mathbb{Z}}\left(1+\frac{n^{2}}{10^{10}}\right)^{-\lambda}\simeq B\left(\frac{1}{2},\lambda-\frac{1}{2}\right)\label{eq:Student sum}
\end{equation}
with an error of the order of $10^{-136440}$ in the worst case scenario,
which corresponds to the value $\lambda=\pi\times 10^5 \simeq 314159.$
\end{proposition}

The accuracy of this approximation can be evaluated as follows: starting
from the modular property of the Jacobi $\theta_{3}$ distribution (see \cite[20.2.3]{NIST} for its definition)
\[
\sum_{n\in\mathbb{Z}}e^{-t\frac{n^{2}}{a^{2}}}=a\sqrt{\frac{\pi}{t}}\sum_{n\in\mathbb{Z}}e^{-\frac{\pi^{2}n^{2}a^{2}}{t}},
\]
multiplying by $t^{\lambda-1}e^{-t}$ and integrating over $\mathbb{R}^{+}$
(which means that we compute the Mellin-Laplace transform of $\exp(-t\ n^{2}/a^{2})\ $) on both sides now produces
\begin{equation}
\label{Poisson}
\frac{1}{a}\sum_{n\in\mathbb{Z}}\left(1+\frac{n^{2}}{a^{2}}\right)^{-\lambda}=\frac{2^{\frac{3}{2}-\lambda}\sqrt{\pi}}{\Gamma\left(\lambda\right)}\sum_{n\in\mathbb{Z}}\left(2\pi a\vert n\vert\right)^{\lambda-\frac{1}{2}}K_{\lambda-\frac{1}{2}}\left(2\pi a\vert n\vert\right),
\end{equation}
where $K_{\lambda}$ is the modified Bessel function of the second
kind. Since we have the Fourier transform \cite[2.3.5.7]{Prudnikov}
\[
\int_{-\infty}^{\infty} 
\left(1+\frac{x^2}{a^2}\right)^{-\lambda}e^{-\imath 2 \pi x z}dx 
=a\frac{2^{\frac{3}{2}-\lambda}\sqrt{\pi}}{\Gamma\left(\lambda\right)}\left(2\pi a\vert z\vert\right)^{\lambda-\frac{1}{2}}K_{\lambda-\frac{1}{2}}\left(2\pi a\vert z\vert\right),
\]
\sloppy identity \eqref{Poisson} is recognized as the Poisson summation formula $\sum_{n=-\infty}^{\infty}s\left(n\right) = \sum_{n=-\infty}^{\infty}S\left(n\right)$
associated to the sequence 
$
s\left(n\right)=\left(1+\frac{n^2}{a^2}\right)^{-\lambda }
$
and its Fourier transform $S.$
The Bessel $K_{\lambda}$ function behaves asymptotically \cite[10.30.2 and 10.40.2]{NIST} like $\lim_{x\to0}x^{\nu}K_{\nu}\left(x\right)=2^{\nu-1}\Gamma\left(\nu\right)$ and, at infinity, $x^{\nu}K_{\nu}\left(x\right)\sim\sqrt{\frac{\pi}{2}}z^{\nu-\frac{1}{2}}e^{-z}$,
so that, for a large value of the parameter $a,$ 
\[
\frac{1}{a}\sum_{n\in\mathbb{Z}}\left(1+\frac{n^{2}}{a^{2}}\right)^{-\lambda}-B\left(\frac{1}{2},\lambda-\frac{1}{2}\right)\simeq\frac{2}{a}\frac{\left(\pi a\right)^{\lambda}}{\Gamma\left(\lambda\right)}e^{-2\pi a}.
\]
The mapping $\lambda\mapsto\frac{\left(\pi a\right)^{\lambda}}{\Gamma\left(\lambda\right)}$ is increasing over $\left(0,\lambda_{*}\right],$ with $\lambda_{*}$ defined as the unique solution
to the equation $\psi\left(\lambda_{*}\right)=\log\left(\pi a\right)$ (with $\psi\left(x\right)=\frac{\Gamma'\left(x\right)}{\Gamma\left(x\right)}$ the digamma function), and decreasing over $\left[\lambda_{*},\infty\right).$ 
Its maximum value, reached at $\lambda = \lambda_{*},$ is equal to
$
\frac{\left(\pi a\right)^{\lambda_{*}}}{\Gamma\left(\lambda_{*}\right)}.
$
For a sufficiently large value of the parameter $a$ (such as $a=10^{5}$), 
$\psi\left(\pi a\right)\simeq\log\left(\pi a\right)
$
so that 
$\lambda_{*}\simeq \pi a$ and, using Stirling's formula \cite[5.11.3]{NIST}, $\frac{\left(\pi a\right)^{\lambda_{*}}}{\Gamma\left(\lambda_{*}\right)}\simeq \sqrt{\frac{a}{2}} e^{\pi a}$, the error term is close to $\sqrt{\frac{2}{a}} e^{-\pi a}$. In the special case $a=10^5$,
the worst-case approximation error can be estimated as 
\begin{align*}
\frac{1}{10^{5}}\sum_{n\in\mathbb{Z}}\left(1+\frac{n^{2}}{10^{10}}\right)^{-\lambda}-B\left(\frac{1}{2},\lambda-\frac{1}{2}\right) & \simeq e^{-314165}\simeq 10^{-136440}.
\end{align*}
For example, in the case $\lambda=5,$
\[
\frac{1}{10^{5}}\sum_{n\in\mathbb{Z}}\left(1+\frac{n^{2}}{10^{10}}\right)^{-5}-B\left(\frac{1}{2},\frac{9}{2}\right) \simeq 2.2\times 10^{-272856}.
\]
As a final remark, the series (\ref{eq:normal sum})
and (\ref{eq:Student sum}) exhibited here are as slowly converging
as they are accurate, all the more as the scaling factor $a$ increases.
It can be experimentally checked that faster converging series are obtained for lower values of the parameter $a$ but at the cost of lower accuracy. 

\section{Zeta Function and Infinite Product of Matrices.}
\label{sec:zeta}
In the fascinating book \textit{Mathematical Constants} \cite{Finch},
S.R. Finch cites this unpublished result by W. Gosper \cite{Gosper}: 
\begin{equation}
\prod_{k=1}^{\infty}\left[\begin{array}{cc}
-\frac{k}{2\left(2k+1\right)} & \frac{5}{4k^{2}}\\
0 & 1
\end{array}\right]=\left[\begin{array}{cc}
0 & \zeta\left(3\right)\\
0 & 1
\end{array}\right],\label{eq:product 2}
\end{equation}
and, for $N\ge2,$ its
extension to the  $\left(N+1\right)\times\left(N+1\right)$ dimensional case

\setlength{\arraycolsep}{1pt}
\begin{align}
\hspace{-0.089872cm}\prod_{k=1}^{\infty}&\left[\begin{array}{cccccc}
\frac{-k}{2\left(2k+1\right)} & \frac{1}{2k\left(2k+1\right)} & 0 & \dots & 0 & \frac{1}{k^{2N}}\\
0 & \frac{-k}{2\left(2k+1\right)} & \frac{1}{2k\left(2k+1\right)} & \dots &  & \frac{1}{k^{2N-2}}\\
\vdots & \vdots & \vdots &  & \vdots & \vdots\\
0 & 0 & 0 & \dots & \frac{1}{2k\left(2k+1\right)} & \frac{1}{k^{4}}\\
0 & 0 & 0 & \dots & \frac{-k}{2\left(2k+1\right)} & \frac{5}{4k^{2}}\\
0 & 0 & 0 & \dots & 0 & 1
\end{array}\right]=\left[\begin{array}{cccc}
0 & \dots & 0 & \zeta\left(2N+1\right)\\
0 & \dots & 0 & \zeta\left(2N-1\right)\\
\vdots &  & \vdots & \vdots\\
0 & \dots & 0 & \zeta\left(5\right)\\
0 & \dots & 0 & \zeta\left(3\right)\\
0 & \dots & 0 & 1
\end{array}\right].
\label{eq:product n}
\end{align}
\setlength{\arraycolsep}{5pt}
In \cite{Gosper}, Gosper also gave a striking infinite matrix product representation of the beta function at the positive even integers which follows a similar pattern. The technique that he used to uncover these identities was partially brute-force computation by means of a modern symbolic processor, a technique not available to the pioneers of the past. Yet, these matrix product representations are quite useful, as it turns out to be more computationally efficient to approximate the value of zeta by means of a continuous pairwise multiplication of a finite number of the matrices from \eqref{eq:product n} until only one matrix is left than to compute the same value using the hypergeometric series representation of the zeta function. In fact, in 1985, Gosper used this technique to temporarily steal the record 
for the most digits of $\pi$ computed (with a mind-boggling 17 million digits) from Yasumasa Kanada and his colleagues in Japan.


Both identities in Entry \ref{Infinite matrix} are extensions of identities \eqref{eq:product 2} and \eqref{eq:product n} to the case of the zeta function computed at even arguments. Unexpectedly, the second part of Entry~\ref{Infinite matrix} turns out to be false, and a correction term involving a hyper-harmonic number is needed to fix it.
\begin{proposition}
The identity 
\begin{equation}
\prod_{n\ge1}\left[\begin{array}{ccc}
\frac{n}{2\left(2n+1\right)} & \frac{-3}{2n\left(2n+1\right)} & \frac{3}{2n^{3}}\\
0 & \frac{n}{2\left(2n+1\right)} & \frac{3}{2n}\\
0 & 0 & 1
\end{array}\right]=\left[\begin{array}{ccc}
0 & 0 & \zeta\left(4\right)\\
0 & 0 & \zeta\left(2\right)\\
0 & 0 & 1
\end{array}\right]\label{eq:zeta(4)}
\end{equation}
is  true while the identity
\begin{equation}
\prod_{n\ge1}\left[\begin{array}{cccc}
\frac{n}{2\left(2n+1\right)} & \frac{-3}{2n\left(2n+1\right)} & 0 & \frac{3}{2n^{5}}\\
0 & \frac{n}{2\left(2n+1\right)} & \frac{-3}{2n\left(2n+1\right)} & \frac{3}{2n^{3}}\\
0 & 0 & \frac{n}{2\left(2n+1\right)} & \frac{3}{2n}\\
0 & 0 & 0 & 1
\end{array}\right]=\left[\begin{array}{cccc}
0 & 0 & 0 & \zeta\left(6\right)\\
0 & 0 & 0 & \zeta\left(4\right)\\
0 & 0 & 0 & \zeta\left(2\right)\\
0 & 0 & 0 & 1
\end{array}\right]\label{eq:zeta(6)}
\end{equation}
is false. 
A correct version of \eqref{eq:zeta(6)} is 
\begin{equation}
\label{eq:zeta(6) correct version}
\hspace{-2mm}\prod_{n\ge1}\left[\begin{array}{cccc}
\frac{n}{2\left(2n+1\right)} & \frac{-3}{2n\left(2n+1\right)} & 0 & \frac{3}{2n^{5}}-\frac{9H_{n-1}^{\left(4\right)}}{2n}\\
0 & \frac{n}{2\left(2n+1\right)} & \frac{-3}{2n\left(2n+1\right)} & \frac{3}{2n^{3}}\\
0 & 0 & \frac{n}{2\left(2n+1\right)} & \frac{3}{2n}\\
0 & 0 & 0 & 1
\end{array}\right]=\left[\begin{array}{cccc}
0 & 0 & 0 & \zeta\left(6\right)\\
0 & 0 & 0 & \zeta\left(4\right)\\
0 & 0 & 0 & \zeta\left(2\right)\\
0 & 0 & 0 & 1
\end{array}\right]    
\end{equation}
with the hyper-harmonic number defined for $n\ge 1$ by $H_n^{\left(p\right)}:=\sum_{k=1}^n \frac{1}{k^p}$.
\end{proposition}

Both identities \eqref{eq:zeta(4)} and \eqref{eq:zeta(6) correct version} 
are based on the following identity by Borwein et al. \cite[Thm. 1.1]{BorweinZeta} that produces 
a generating function for the values of the Riemann zeta function at even arguments:
\begin{equation}
\sum_{n\ge1}\frac{1}{n^{2}-z^2}=3\sum_{k\ge1}\frac{1}{\binom{2k}{k}}\frac{1}{k^{2}-z^2}\prod_{j=1}^{k-1}\frac{j^{2}-4z^2}{j^{2}-z^2}.\label{eq:Borweinzeta}
\end{equation}
This reappearance of the Borwein family in the context of yet another problem was an unexpected surprise that puts their mathematical prowess on display. 

In order to express \eqref{eq:Borweinzeta} as an infinite product of matrices, we notice that each $\left(N+1\right)\times \left(N+1\right)$ matrix in \eqref{eq:zeta(4)} and \eqref{eq:zeta(6) correct version}  is of the form
\[
\M_{k}=\left[\begin{array}{cc}
\A_{k} & \mathbf{u}_{k}\\
\mathbf{0} & 1
\end{array}\right],
\]
where $\A_{k}$ is a square $\left(N\times N\right)$ matrix, $\mathbf{u}_{k} = \left[u_k^{(N)},\dots , u_k^{(1)}\right]^t$ is an
$\left(N\times1\right)$ vector and $\mathbf{0}$ is the $\left(1\times N\right)$ vector of zeros. Moreover, each matrix $\A_{k}$ has the simple form $\A_{k}=\alpha_{k}\mathbf{I}+\beta_{k}\mathbf{J}$, where $\mathbf{I}$ is the $\left(N\times N\right)$ identity matrix and  $\mathbf{J}$ is the $\left(N\times N\right)$ matrix with all entries equal to $0$, with the exception of a first superdiagonal
of ones. Assuming convergence of the infinite product, we deduce the following:
\begin{proposition}
\label{mainlem}
The vector $\mathbf{v}_\infty$ such that
\begin{equation}
\prod_{k=1}^{\infty}\left[\begin{array}{cc}
\A_k & \mathbf{u}_k\\
\mathbf{0} & 1
\end{array}\right]=\left[\begin{array}{cc}
\prod_{k=1}^\infty \A_k & \mathbf{v}_\infty\\
\mathbf{0} & 1
\end{array}\right],\label{eq:product 3}
\end{equation}
satisfies
\begin{equation}
\label{vinfty}    
\mathbf{v}_{\infty}=\left[\begin{array}{c}
v_{\infty}^{\left(N\right)}\\
\vdots\\
v_{\infty}^{\left(1\right)}
\end{array}\right]=\sum_{p=1}^\infty\A_{1}\dots \A_{p-1}\mathbf{u}_{p}.
\end{equation}
Its $\ell$-th component $v_{\infty}^{\left(\ell\right)}$, $1 \le \ell \le N,$ is given by
\begin{align*}
\sum_{p=1}^\infty\left(u_{p}^{\left(\ell\right)}+\left(\sum_{j=1}^{p-1}\frac{\beta_{j}}{\alpha_{j}}\right)u_{p}^{\left(\ell-1\right)}+\cdots+\left(\sum_{\substack{j_1,\ldots,j_{\ell-1}\\1\le j_{i}<j_{i+1}\le p-1}}
\prod_{q=1}^{\ell-1}\frac{\beta_{j_{q}}}{\alpha_{j_{q}}}\right)
u_{p}^{\left(1\right)}\right)\prod_{q=1}^{p-1}\alpha_{q},
\end{align*}
with the first cases
\begin{align*}
v_{\infty}^{\left(1\right)}&=\sum_{p=1}^\infty\left(\alpha_{1}\cdots\alpha_{p-1}\right)u_{p}^{\left(1\right)},\\
v_{\infty}^{\left(2\right)}&=\sum_{p=1}^\infty\left(\alpha_{1}\cdots\alpha_{p-1}\right)\left(u_{p}^{\left(2\right)}+\left(\sum_{j=1}^{p-1}\frac{\beta_{j}}{\alpha_{j}}\right)u_{p}^{\left(1\right)}\right). \nonumber\\
\end{align*}
\end{proposition}
Notice that \eqref{vinfty} is the asymptotic version of the more general identity, with $n \ge 1,$ $\mathbf{v}_{n}=\sum_{i=1}^{n}A_{1}\dots A_{i-1}\mathbf{u}_{i}$. From Proposition \ref{mainlem}, we deduce
\begin{proposition}
An infinite matrix product representation for $\zeta\left(2\right)$ and $\zeta\left(4\right)$
is 
\[
\prod_{n\ge1}\left(\begin{array}{ccc}
\frac{n}{2\left(2n+1\right)} & \frac{-3}{2n\left(2n+1\right)} & \frac{3}{2n^{3}}\\
0 & \frac{n}{2\left(2n+1\right)} & \frac{3}{2n}\\
0 & 0 & 1
\end{array}\right)=\left(\begin{array}{ccc}
0 & 0 & \zeta\left(4\right)\\
0 & 0 & \zeta\left(2\right)\\
0 & 0 & 1
\end{array}\right).
\]
\end{proposition}

\begin{proof}
With the notations above, this is the case
$\alpha_{n}=\frac{n}{2\left(2n+1\right)},$ $\beta_n=\frac{-3}{2n(2n+1)}$ and $\mathbf{u}_{n}=\left[\begin{array}{c}
\frac{3}{2n^3}\ ,\ \frac{3}{2n}
\end{array}\right]^t,$
from which we deduce
\begin{align*}
\mathbf{v}_{\infty} & =
\left[\begin{array}{c}
\sum_{i=1}^{\infty}\frac{2}{i\binom{2i}{i}}\left(\frac{3}{2i^3}-\sum_{j=1}^{i-1}\frac{3}{j^2}\frac{3}{2i}\right)\\
\sum_{i=1}^{\infty}\frac{2}{i\binom{2i}{i}}\frac{3}{2i}
\end{array}\right].
\end{align*}
Using Markov's identity (see \cite{Kondratieva} for an historical perspective on this identity) 
\[\zeta\left(2\right)=\sum_{i=1}^{\infty}\frac{3}{i^{2}\binom{2i}{i}},\]
the first component of $\mathbf{v}_{\infty}$ is identified as
$v_{\infty}^{\left(1\right)}=\zeta\left(2\right).$ Moreover, it can be easily checked that $\prod_{k=1}^\infty \A_k$ equals the $N\times N$ null matrix.
Notice that Markov's identity can be obtained by identifying the constant term on both sides of Borwein's generating function \eqref{eq:Borweinzeta}. 
Moreover,
\[
v_{\infty}^{\left(2\right)}=\sum_{i=1}^{\infty}\frac{3}{2i^3\binom{2i}{i}}+\sum_{i=1}^{\infty}\frac{3}{2i^2\binom{2i}{i}}\sum_{j=1}^{i-1}\frac{\beta_{j}}{\alpha_{j}}=\zeta\left(4\right)
\]
by identifying the coefficient of $z^2$ in Borwein's identity \eqref{eq:Borweinzeta}.
\end{proof}
Unfortunately, the extension of the previous representation to $\zeta\left(6\right)$ is not
as straightforward.
\begin{proposition}
An infinite matrix product 
representation for $\zeta\left(2\right),\thinspace\thinspace\zeta\left(4\right)$
and $\zeta\left(6\right)$ is 
\[
\prod_{n\ge1}\left[\begin{array}{cccc}
\frac{n}{2\left(2n+1\right)} & -\frac{3}{2n\left(2n+1\right)} & 0 & \frac{3}{2n^{5}}-\frac{9H_{n-1}^{\left(4\right)}}{2n}\\
0 & \frac{n}{2\left(2n+1\right)} & -\frac{3}{2n\left(2n+1\right)} & \frac{3}{2n^{3}}\\
0 & 0 & \frac{n}{2\left(2n+1\right)} & \frac{3}{2n}\\
0 & 0 & 0 & 1
\end{array}\right]=\left[\begin{array}{cccc}
0 & 0 & 0 & \zeta\left(6\right)\\
0 & 0 & 0 & \zeta\left(4\right)\\
0 & 0 & 0 & \zeta\left(2\right)\\
0 & 0 & 0 & 1
\end{array}\right].
\]
\end{proposition}

\begin{proof}
Identifying the coefficient of $z^{4}$ in Borwein's identity (\ref{eq:Borweinzeta})
produces
\[
\zeta\left(6\right)=3\sum_{k\ge1}\frac{1}{\binom{2k}{k}k^{2}}\left[17H_{k-1}^{\left(2,2\right)}+H_{k-1}^{\left(4\right)}-4\left(H_{k-1}^{\left(2\right)}\right)^{2}-\frac{3H_{k-1}^{\left(2\right)}}{k^{2}}+\frac{1}{k^{4}}\right].
\]
with the bivariate hyper-harmonic sums
\[
H_{n}^{\left(p,q\right)}=\sum_{k_1>k_2\ge1}\frac{1}{k_1^{p}k_2^{q}}.
\]
Moreover, the vector $\mathbf{v}_{n}=
\sum_{i=1}^{n}A_{1}\cdots A_{i-1}\mathbf{u}_{i}$ is computed as 
\[
\mathbf{v}_{n}=\sum_{i=1}^{n}\frac{2}{\binom{2i}{i}i}\left[\left[\begin{array}{c}
u_{i}^{\left(3\right)}\\
u_{i}^{\left(2\right)}\\
u_{i}^{\left(1\right)}
\end{array}\right]
\hspace{-0.1cm}
-3H_{i-1}^{\left(2\right)}\left[\begin{array}{c}
u_{i}^{\left(2\right)}\\
u_{i}^{\left(1\right)}\\
0
\end{array}\right]+9H_{i-1}^{\left(2,2\right)}\left[\begin{array}{c}
u_{i}^{\left(1\right)} \\
0\\
0
\end{array}\right]\right].
\]
Hence,
\begin{align*}
v_{\infty}^{\left(3\right)} & 
  =\sum_{i=1}^{\infty}\frac{2}{\binom{2i}{i}i}\left[u_{i}^{\left(3\right)}-3H_{i-1}^{\left(2\right)}\frac{3}{2i^{3}}+9H_{i-1}^{\left(2,2\right)}\frac{3}{2i}\right].
\end{align*}
Using the identity $\left(H_{n}^{\left(2\right)}\right)^2=2H_{n}^{\left(2,2\right)}+H_{n}^{\left(4\right)}$, which is a consequence of the more general and elementary identity $\sum_{i<j}a_{i}a_{j}+\sum_{i>j}a_{i}a_{j}+\sum_{i=j}a_{i}a_{j}=\left(\sum_{i}a_{i}\right)^2$,
and identifying $v_{\infty}^{\left(3\right)}=\zeta\left(6\right)$
produces the result.
\end{proof}
Similar computations show that another version of \eqref{eq:zeta(6) correct version} is given by
\begin{equation}
\label{alternate version}
\hspace{-2mm}
\prod_{n\ge1}\left[\begin{array}{cccc}
\frac{n}{2\left(2n+1\right)} & -\frac{3}{2n\left(2n+1\right)} & 0 & \frac{3}{2n^{5}}\\
0 & \frac{n}{2\left(2n+1\right)} & -\frac{3}{2n\left(2n+1\right)} & \frac{3}{2n^{3}}\\
0 & 0 & \frac{n}{2\left(2n+1\right)} & \frac{3}{2n}\\
0 & 0 & 0 & 1
\end{array}\right]
=\left[\begin{array}{cccc}
0 & 0 & 0 & \zeta\left(6\right)+ \delta\\
0 & 0 & 0 & \zeta\left(4\right)\\
0 & 0 & 0 & \zeta\left(2\right)\\
0 & 0 & 0 & 1
\end{array}\right]
\end{equation}
with the constant $\delta=9\sum_{n\ge1}\frac{H_{n-1}^{(4)}}{\binom{2k}{k}k^2}\simeq 0.438668$.
The symmetry between representations \eqref{eq:zeta(6) correct version} and \eqref{alternate version} is puzzling. For a delightful account of the ``computational virtues of matrix products'' and much more, we recommend \cite{Gosper2}. In Gosper's own words \cite[p. 3]{Gosper2}:
\begin{quote}
    ``It is strange that so
fruitful a representation is applied so rarely to sums. Perhaps discouraging is the apparent
nonlinearity of the matrix product form, (partially) concealing such familiar operations as
termwise differentiation or combination with other series. We shall see that a small bag of
tricks, again based on path invariance, more than remedies these problems.''
\end{quote}
\section{Moments of a Discrete Normal Distribution.}
The literature about discrete normal distributions is scarce and quite recent. 
We relied primarily on \cite{Kemp}
while \cite{Agostini} studies the multivariate case and \cite{Romik} exploits their connection with elliptic functions. The discrete normal distribution is important in probability theory because, as shown in \cite{Kemp}, it is the discrete distribution that maximizes Shannon entropy under variance constraint. The situation is analogous to its continuous counterpart, the normal distribution, which maximizes the differential entropy under variance constraint.
\subsection{The Moment of Order Two.}
The first suspicious identity 
\[
\sum_{n\in\mathbb{Z}}n^{2}\exp\left(-\frac{n^{2}}{2}\right)=\sum_{n\in\mathbb{Z}}\exp\left(-\frac{n^{2}}{2}\right),
\]
is inspired by its continuous counterpart 
\[
\int_{-\infty}^{\infty}x^{2}e^{-\frac{x^{2}}{2}}dx=\int_{-\infty}^{\infty}e^{-\frac{x^{2}}{2}}dx=\sqrt{2\pi}.
\]
Numerical evaluation shows indeed that $ \theta_{3}\left(q\right)$ with $q=e^{-\frac{1}{2}}$ is approximated as
\begin{align*}
\sum_{n\in\mathbb{Z}}\exp\left(-\frac{n^{2}}{2}\right) &
 \simeq\sqrt{2\pi}+1.32\times10^{-8},
\end{align*}
while 
\begin{align*}
\sum_{n\in\mathbb{Z}}n^{2}\exp\left(-\frac{n^{2}}{2}\right) 
 & \simeq\sqrt{2\pi}-5.16\times10^{-7}.
\end{align*}
Here $\theta_3$ is the Jacobi theta function defined by \cite[20.2.3]{NIST} $\theta_{3}\left(q\right)=\sum_{n\in\mathbb{Z}}q^{n^2}$. An exact formula for the ratio of these two sums is given by \cite[Eq. (9)]{Romik}
\begin{equation}
\label{sigma2}
\frac{\sum_{n\in\mathbb{Z}}n^{2}\exp\left(-\frac{n^{2}}{2}\right)}{\sum_{n\in\mathbb{Z}}\exp\left(-\frac{n^{2}}{2}\right)}=\frac{K\left(k\right)^{2}}{\pi}\left[\frac{E\left(k\right)}{K\left(k\right)}-k'^{2}\right]\simeq0.9999997887677.
\end{equation}
Here $k\simeq0.99999997859$ is the elliptic modulus associated to $q=\text{exp}(-\pi \frac{K'(k)}{K(k)})=e^{-1/2}$, $k'=\sqrt{1-k^2}\simeq0.00020689274$ is the complementary modulus, while $K$ and $E$ are the complete elliptic integrals of the first and second kind, respectively.  

Reciprocally, one may ask how close to $\frac{1}{2}$ is the value of the parameter $c=\pi \frac{K'(k)}{K(k)}$ such that
\begin{equation}
\frac{\sum_{n\in\mathbb{Z}}n^{2}\exp\left(- c n^{2}\right)}{\sum_{n\in\mathbb{Z}}\exp\left(- c n^{2}\right)}=
1.
\end{equation}
From \eqref{sigma2}, the corresponding elliptic modulus $k$ 
is the unique solution to the equation 
\[
\frac{K^{2}\left(k\right)}{\pi^{2}}\left(\frac{E\left(k\right)}{K\left(k\right)}-k'^{2}\right)=1,
\]
which can not be solved analytically. However, Mathematica produces the numerical
value $k\simeq0.99999997859.$ 
The value of the parameter $c$ is deduced as $c= \frac{K'(k)}{K(k)}\simeq0.49999989438
$
which is close to $\frac{1}{2}$ indeed.
\subsection{The Moment of Order Four}
The analogy with Gaussian integrals continues with the second identity in Entry \ref{discrete normal}. In the continuous case, 
\[
\frac{\int_{-\infty}^{\infty}x^{4}e^{-\frac{x^{2}}{2}}dx}{\int_{-\infty}^{\infty}e^{-\frac{x^{2}}{2}}dx}=3,
\]
while in the discrete case 
\begin{align*}
\frac{\sum_{n\in\mathbb{Z}}n^{4}\exp\left(-\frac{n^{2}}{2}\right)}{\sum_{n\in\mathbb{Z}}\exp\left(-\frac{n^{2}}{2}\right)}\simeq 3.000000707.
\end{align*}
In fact, we have the exact formula \cite[Thm. 3]{Romik}
\[
\frac{\sum_{n\in\mathbb{Z}}n^{4}\exp\left(-\frac{n^{2}}{2}\right)}{\sum_{n\in\mathbb{Z}}\exp\left(-\frac{n^{2}}{2}\right)}=3\sigma^{4}+\frac{1}{8}\theta_{3}^{8}\left(e^{-\frac{1}{2}}\right)\left(kk'\right)^{2}
\]
with 
\[
\sigma^{2}=\frac{\sum_{n\in\mathbb{Z}}n^{2}\exp\left(-\frac{n^{2}}{2}\right)}{\sum_{n\in\mathbb{Z}}\exp\left(-\frac{n^{2}}{2}\right)}
\]
as computed previously in \eqref{sigma2}.
The correction term is approximately equal to $\frac{1}{8}\theta_{3}^{8}\left(e^{-\frac{1}{2}}\right)\left(kk'\right)^{2}\simeq8.33912\times10^{-6}$.

\subsection{The Case of the Jacobi $\theta_2$ Function.}
An extension of the previous results to the Jacobi $\theta_2$ function $\theta_{2}\left(q\right)=\sum_{n\in\mathbb{Z}}q^{(n-\frac{1}{2})^2}$
is given by
\label{discrete normal2}
\[
\sum_{n\in\mathbb{Z}}\left(n-\frac{1}{2}\right)^{2}\exp\left(-\frac{1}{2}\left(n-\frac{1}{2}\right)^{2}\right)\simeq\sum_{n\in\mathbb{Z}}\exp\left(-\frac{1}{2}\left(n-\frac{1}{2}\right)^{2}\right).
\]
Numerical values are \[\sum_{n\in\mathbb{Z}}\exp\left(-\frac{1}{2}\left(n-\frac{1}{2}\right)^{2}\right)  =\theta_{2}\left(q=e^{-\frac{1}{2}}\right)
 \simeq\sqrt{2\pi}-1.34\times10^{-8}\] and 
\[
\sum_{n\in\mathbb{Z}}\left(n-\frac{1}{2}\right)^{2}\exp\left(-\frac{1}{2}\left(n-\frac{1}{2}\right)^{2}\right) 
  \simeq\sqrt{2\pi}+5.16\times10^{-7}.
\]
Similar to the $\theta_3$ case, an exact formula
is available for the ratio
\[
\frac{\sum_{n\in\mathbb{Z}}\left(n-\frac{1}{2}\right)^{2}q^{\frac{1}{2}\left(n-\frac{1}{2}\right)^{2}}}{\sum_{n\in\mathbb{Z}}q^{\frac{1}{2}\left(n-\frac{1}{2}\right)^{2}}}=\frac{1}{\pi^{2}}E\left(k\right)K\left(k\right)-\frac{1}{4}\theta_{4}^{4}\left(q\right)
\]
with $\theta_4$ the theta function \cite[20.2.4]{NIST}
$\theta_{4}\left(q\right)=\sum_{n\in\mathbb{Z}}\left(-1\right)^{n}q^{n^2}$. 
A numerical evaluation in the case $q=e^{-\frac{1}{2}}$ produces
\[
\frac{\sum_{n\in\mathbb{Z}}\left(n-\frac{1}{2}\right)^{2}\exp\left(-\frac{1}{2}\left(n-\frac{1}{2}\right)^{2}\right)}{\sum_{n\in\mathbb{Z}}\exp\left(-\frac{1}{2}\left(n-\frac{1}{2}\right)^{2}\right)}
\simeq1.000000211232.
\]


\section{A Sum Over Integer Partitions.}\label{sec:partitionsum}

Our next identity involves a sum over integer partitions and a binomial coefficient: 
\begin{align}
\sum_{1j_{1}+\cdots+nj_{n}=n}\frac{n^{j_{1}+\cdots+j_{n}}}{\prod_{k=1}^{n}k^{j_{k}}j_{k}!}=\binom{2n-1}{n},\label{eq:integer-partition-sum}
\end{align}
where the $j_i$ are each non-negative integers. As we will see in a moment, this is a welcome identity for combinatorists and quantum information theorists alike. It transforms the computation of an odd looking sum over the partitions of an integer to a much simpler binomial coefficient computation, so that if this strange sum were to appear anywhere practical, the identity \eqref{eq:integer-partition-sum} would serve us well indeed!

Before explaining \eqref{eq:integer-partition-sum}, let us offer some context. A classical problem in combinatorics is to compute the number of
non-equivalent colorings of $n$ objects with $d$ colors under the
assumption that colorings connected by the action of a group $G$
are considered equivalent. For example, one might consider the colorings
of $n$ beads on a necklace using $d$ colors. The cyclic group of
order $n$ acts on the set of beads in a natural way corresponding
to the rotations of the necklace. A solution to the general problem
is given by the enumeration theorem of the Hungarian mathematician G. P\'olya, who is well-known for his contributions to combinatorics, among other fields. The theorem gives the number of orbits of the action of $G$ on a set of colorings in terms of the cycle index polynomial $Z(G)(x_{1},\ldots,x_{n})=\frac{1}{\lvert G\rvert}\sum_{g\in G}x_{1}^{c_{1}(g)}\cdots x_{n}^{c_{n}(g)}$ \cite{Polya},
where $c_{j}(g)$ denotes the number of cycles of order $j$ in the
disjoint cycle decomposition of $g$.
\begin{theorem}[P\'olya's Enumeration Theorem]
Let $G$ be a finite group (which we can take to be a permutation
group by Cayley's theorem) with letters in a set $X$. Let $C$ be
a set with $\lvert C\rvert=d$ and suppose $\Omega$ is the set of
all functions $f:X\to C$. Then $G$ acts naturally on $\Omega$ and
the number of orbits of $G$ acting on $\Omega$ is given by 
\begin{align}
Z(G)(d,\ldots,d),\label{eq:cycleindex}
\end{align}
where $Z(G)$ denotes the cycle index polynomial of $G$. 
\end{theorem}

In the context of our colorings, $X$ consists of the objects to be
colored, $C$ consists of the set of colors, and $\Omega$ is then
the set of all colorings of $X$, which the chosen group $G$ then
acts on naturally.
Therefore, the number of nonequivalent colorings
under the action of $G$ is given by the cycle index polynomial as
in \eqref{eq:cycleindex}. It is a standard result that the cycle
index polynomial of the symmetric group $S_{n}$ is given by 
\begin{align*}
\sum_{j_{1}+\cdots+nj_{n}=n}\prod_{k=1}^{n}\frac{x_{k}^{j_{k}}}{k^{j_{k}}j_{k}!}.
\end{align*}
Thus, the number of nonequivalent colorings
of $n$ objects with $d$ colors under the action of the symmetric
group is given by 
\begin{align*}
\sum_{j_{1}+\cdots+nj_{n}=n}\frac{d^{j_{1}+\cdots+j_n}}{\prod_{k=1}^{n}k^{j_{k}}j_{k}!}.
\end{align*}
We will now show that there is a convenient closed form for this strange
sum. 
\begin{proposition}
It holds that 
\begin{align}
\sum_{j_{1}+\cdots+nj_{n}=n}\frac{d^{j_{1}+\cdots+j_n}}{\prod_{k=1}^{n}k^{j_{k}}j_{k}!}=\binom{n+d-1}{n}.\label{eq:cycleindexsum}
\end{align}
\end{proposition}

\begin{proof}
We first show that $f(t)=e^{d\sum_{k=1}^{\infty}t^{k}/k}$ 
is a generating function for the partition sum in \eqref{eq:cycleindexsum}.
Indeed, we have 
\begin{align*}
f(t) & =e^{d\sum_{k=1}^{\infty}\frac{t^{k}}{k}}=\prod_{k=1}^{\infty}e^{d\frac{t^{k}}{k}}=\prod_{k=1}^{\infty}\sum_{j_{k}=0}^{\infty}\frac{d^{j_{k}}t^{kj_{k}}}{k^{j_{k}}j_{k}!}.
\end{align*}
Then the $n$-th coefficient is given by the sum over all terms where
the exponent $\sum_{k=1}^{\infty}kj_{k}$ of $t$ is equal to $n$, each
of which is given by a partition of $n$. That is, we have that the
$n$-th coefficient is precisely 
\begin{align*}
\sum_{j_{1}+\cdots+nj_{n}=n}\frac{d^{j_{1}+\cdots+j_n}}{\prod_{k=1}^{n}k^{j_{k}}j_{k}!}.
\end{align*}
On the other hand, by recognizing the Taylor expansion of $\log(1-t)$ in the exponents, we find that $f(t)=\frac{1}{(1-t)^{d}}$, 
which has $n$-th coefficient $\binom{d+n-1}{n}$, 
and this completes the proof.
\end{proof}
If $n$ is large, computing all possible
partitions of $n$, as required by the summation, can be cumbersome,
but the closed form is readily evaluated. For example, for $n=100$ and $d=101,$ we have 
\begin{align*}
    \sum_{j_{1}+\cdots+100j_{100}=100}\frac{101^{j_{1}+\cdots+j_{100}}}{\prod_{k=1}^{100}k^{j_{k}}j_{k}!}=\binom{100+101-1}{100},
\end{align*} which is easily evaluated by a computer, producing the exact value $90548514656103281165404177077484163874504589675413336841320$.
Moreover, when $d=n$, we recover \eqref{eq:integer-partition-sum}, and this explains Entry~\ref{coloring}. 

Note that the cycle index polynomial of the alternating group $A_{n}$ is 
\begin{align*}
\sum_{j_{1}+\cdots+nj_{n}=n}\frac{1+(-1)^{j_2+j_4+\cdots}}{\prod_{k=1}^{n}k^{j_{k}}j_{k}!}\prod_{k=1}^{n}x_{k}^{j_{k}},
\end{align*}
where $j_2+j_4+\cdots$ denotes the sum of the indices with even subscripts, and a similar generating function argument, where the generating function is now $f(t)=e^{d\sum_{k=1}^{\infty}t^{k}/k}+e^{-d\sum_{k=1}^{\infty}(-t)^{k}/k}$, produces
\begin{align*}
\sum_{j_{1}+\cdots+nj_{n}=n}
\hspace{-3mm}
\frac{1+(-1)^{j_2+j_4+\cdots}}{\prod_{k=1}^{n}k^{j_{k}}j_{k}!}d^{j_{1}+\cdots+j_n}
\hspace{-1mm}
=
\hspace{-1mm}
\binom{n+d-1}{n}
\hspace{-1mm}
+
\frac{d(d-1)\cdots(d-n+1)}{n!}.
\end{align*}
From the perspective of coloring theory, we see that the difference between the number of colorings with respect to the actions of the alternating and symmetric groups is precisely $d(d-1)\cdots(d-n+1)/n!$. Therefore, there are just as many colorings with respect to the action of the symmetric group as there are with respect to the alternating group whenever the number of objects to be colored is greater than the number of colors; that is, when $n>d$.

Of course, P\'olya's enumeration theorem is much more general than coloring theory, and there are many applications of this result outside of that context. An application to linear algebra, for example, is given by considering the action of a finite permutation group $G$ on a tensor product space $V\otimes\cdots\otimes V$ by permuting the constituent spaces according to the chosen permutation. That is, with $\sigma\in G\subset S_n$ and $\{e_i\}_{i=1}^d$ a basis for $V$, we define $\sigma\cdot (e_{i_1}\otimes\cdots\otimes e_{i_n})= e_{i_{\sigma(1)}}\otimes\cdots\otimes e_{i_{\sigma(n)}}$ and extend linearly. Then $\Pi_G=(1/\lvert G\rvert)\sum_{\sigma\in G}\sigma$ is a projection operator and its image has dimension $Z(G)(d,\ldots,d)$. Thus, the dimensions of the subspaces corresponding to the actions of $S_n$ and $A_n$ are given by the identities we have derived in this section. These subspaces appear in quantum computation and quantum information theory in a class of algorithms that test for entanglement in a pure quantum state \cite{Bradshaw2,LaBorde}.

\section{A Digital High Precision Fraud.}\label{sec:hpfraud}

Entry~\ref{digits} is a high precision fraud along the lines of Sum 4 in \cite{Borwein}. This sum is unusual for two reasons. For one, the summand contains a peculiar combination of digital sequences that only a mathematician brave enough to produce a dictionary of the real numbers \cite{BBDictionary} would think to sum. On the other hand, the (fraudulent) result of the summation is quite nice looking. In fact, we will see that it gets even better. This unconventional almost-identity can be extended in an unexpected way, but first let us explain how this special case comes about.

Let $a(n)$ denote the number of even digits in the
decimal expansion of $n$ (entry A196563 in the OEIS) and let $b(n)$ denote the number of odd
digits (entry A196564 in the OEIS). Define the sequence $c(n):=10^{5}a(n)-b(n)/10^{5}$. Then
we have the fraudulent identity 
\begin{align*}
\sum_{n=0}^{\infty}\frac{c(n)}{10^{n}}
=\frac{11111111111}{110000},
\end{align*}
which is correct up to 105 decimal places. To see this, we will need
the generating functions for both $a(n)$ and $b(n)$. There is a nice derivation of the generating function for $b(n)$ in \cite{Borwein}, and we include it here for completeness. We start with the expression
\begin{align*}
    \sum_{n=0}^\infty t^{b(n)}x^n=\prod_{n=0}^\infty\left(1+tx^{10^n}+x^{2\cdot 10^n}+tx^{3\cdot10^n}+\cdots+x^{8\cdot 10^n}+tx^{9\cdot10^n}\right),
\end{align*}
which is easily checked by comparing the coefficients of $x^m$ on both sides. Indeed, on the right hand side, we pick up a factor of $t$ every time $m$ has an odd digit. Now taking the logarithmic derivative in $t$, we come to the expression
\begin{align*}
    \frac{\sum_{n=0}^\infty b(n)t^{b(n)-1}x^n}{\sum_{n=0}^\infty t^{b(n)}x^n}=\sum_{n=0}^\infty\frac{x^{10^n}+x^{3\cdot 10^n}+x^{5\cdot 10^n}+x^{7\cdot 10^n}+x^{9\cdot 10^n}}{1+tx^{10^n}+x^{2\cdot 10^n}+\cdots+x^{8\cdot 10^n}+tx^{9\cdot10^n}},
\end{align*}
and letting $t=1$ yields
\begin{align} \label{eq:odd-digits-generating-function}
\sum_{n=0}^{\infty}b(n)x^{n}=\frac{1}{1-x}\sum_{n=0}^{\infty}\frac{x^{10^{n}}}{1+x^{10^{n}}}.
\end{align}
To derive the generating function for $a(n)$, we will find the generating function for the number of digits $d(n)$ in the decimal representation of $n$ and note $a(n)=d(n)-b(n)$. We claim that the generating function for $d(n)$ is given by
\begin{align} \label{eq:digits-generating-function}
    \sum_{n=0}^\infty d(n)x^n=1+\frac{1}{1-x}\sum_{n=0}^\infty x^{10^n}.
\end{align}
To see this, we must check that the coefficients of $x^m$ agree for all $m$. This is clearly true for $m=0$ since $d(0)=1$. Letting $m>0$, we have
\begin{align*}
    \frac{1}{m!}\frac{d^m}{dx^m}\bigg\lvert_{x=0}
    \hspace{-1mm}
    \left(\hspace{-1mm}1+\hspace{-1mm}\sum_{n=0}^\infty \frac{x^{10^n}}{1-x}\hspace{-1mm}\right)&
    =
    \hspace{-1mm}
    \bigg[\frac{1}{m!}\sum_{k=0}^m\binom{m}{k}\frac{d^{m-k}}{dx^{m-k}}\frac{1}{1-x}\frac{d^k}{dx^k}\sum_{n=0}^\infty x^{10^n}\bigg]_{x=0}\\
    &=\bigg[\sum_{k=0}^m\frac{1}{k!}\frac{1}{(1-x)^{m-k+1}}\frac{d^k}{dx^k}\sum_{n=0}^\infty x^{10^n}\bigg]_{x=0}.
\end{align*}
In the remaining derivative, we see that evaluating at $x=0$ will give zero unless $k$ is precisely a power of 10, in which case, the derivative yields $(10^r)!$ for some $r$. Thus, we need only sum over the powers of ten less than or equal to $m$, giving us
\begin{align*}
    \frac{1}{m!}\frac{d^m}{dx^m}\bigg\lvert_{x=0}\left(1+\frac{1}{1-x}\sum_{n=0}^\infty x^{10^n}\right)&=d(m),
\end{align*}
and so \eqref{eq:digits-generating-function} holds. It follows that the generating function for $a(n)$ is given by $\sum_{n=0}^{\infty}a(n)x^{n}=1+\frac{1}{1-x}\sum_{n=0}^{\infty}\frac{x^{2\cdot10^{n}}}{1+x^{10^{n}}}$, and so the generating function for $c(n)$ is
\begin{align}
\sum_{n=0}^{\infty}c(n)x^{n}=10^{5}+\frac{1}{1-x}\sum_{n=0}^{\infty}\frac{10^{5}x^{2\cdot10^{n}}-x^{10^{n}}/10^{5}}{1+x^{10^{n}}}.\label{eq:generating-function-c}
\end{align}
Evaluating \eqref{eq:generating-function-c} at $x=1/10$, we notice that the second term $(n=1)$ in the right sum vanishes
and the following terms $(n>1)$ become increasingly small. In particular,
we have $\sum_{n=0}^{\infty}\frac{c(n)}{10^{n}}=\frac{11111111111}{110000}-\varepsilon$,
where $\varepsilon<10^{-105}$.

Notice that the choice of $10^5$ in $c(n)$ is not necessary. Indeed, we can generalize the above result by defining $c_k(n):=k^5a(n)-b(n)/k^5$ with $k>1$. Then
\begin{align*}
\sum_{n=0}^{\infty}c_k(n)x^{n}=k^5+\frac{1}{1-x}\sum_{n=0}^{\infty}\frac{k^5x^{2\cdot10^{n}}-x^{10^{n}}/k^5}{1+x^{10^{n}}}
\end{align*}
and evaluating at $x=1/k$ produces
\begin{align}\label{eq:plugged-in}
    \sum_{n=0}^\infty\frac{c_k(n)}{k^n}=k^5+\frac{k}{k-1}\sum_{n=0}^\infty\frac{k^{5-2\cdot10^n}-k^{-5-10^n}}{1+k^{-10^n}},
\end{align}
so that the second term in the right sum vanishes and the remaining terms again become increasingly small. The sum can therefore be approximated by $k^5+\frac{k^4-k^{-5}}{k-\frac1k}$. For example, with $k=100$, we have
\begin{align*}
    \sum_{n=0}^\infty\frac{c_{100}(n)}{100^n}\simeq\frac{101010101010101010101}{10100000000}
\end{align*}
with an error on the order of $10^{-210}$. In fact, this nice pattern involving the powers of ten continues with the number of 0's between 1's increasing by one at each step. That is, the sum $\sum_{n=0}^\infty \frac{c_{10^p}(n)}{10^{pn}}$ is approximated by
\begin{align*}
    \frac{10
    \cdots010\cdots010\cdots010\cdots010\cdots010\cdots010\cdots010\cdots010\cdots010\cdots01}{10\cdots010\cdots0},
\end{align*}
with $p-1$ zeros between each consecutive pair of the eleven 1's in the numerator, as well as the two in the denominator, and $4p$ zeros at the end of the denominator. Written more concisely, we have $\sum_{n=0}^\infty\frac{c_{10^p}(n)}{10^{pn}}\simeq\frac{[1[0]_{p-1}]_{10}1}{1[0]_{p-1}1[0]_{4p}}$, where $[x]_k:=x\cdots x$ denotes $k$ copies of the symbol $x$. The error in this approximation is on the order of $10^{-105p}$. To see this, note that the $n=2$ term in \eqref{eq:plugged-in} with $k=10^p$ is $\frac{10^p(10^{-195p}-10^{-105p})}{(1+10^{-100p})(10^p-1)}$, which is indeed on the order of $10^{-105p}$, and the remaining terms are much smaller.

The approximations exhibited in this section fall under the category of 'digital sums,' which are sums of sequences involving the digits of an integer in some base. The method we employed to
construct these high precision frauds involved combining two generating functions which behave similarly in such a way that the magnitude of the resulting sum is mostly concentrated in its first term. We suspect that there are many more examples following this line of reasoning, and we leave it to the reader to come up with one of their own.

\section{Conclusion.} 
We have posed several puzzling identities, some of which are high precision frauds, and we hope that the reader will feel inspired to add to this zoo of dubious identities. 
Our main goal here was to extend the domain of the identities derived in \cite{Borwein} to other mathematical areas such as combinatorics or elliptic functions.

We will conclude with a few remarks about some of the identities we have discussed here, including possible future directions. Entry \ref{coloring} is related to the theory of colorings modulo the action of the alternating and symmetric groups. We wonder if considering the action of a different finite group will lead to a similar result. Both identities in Entry \ref{Infinite matrix} should be extended to other zeta functions. The reference \cite{Wakhare 3} contains extensions of such identities to the alternating zeta function
$\tilde{\zeta}(2)$ and $\tilde{\zeta}(4)$ with $\tilde{\zeta}(s)=\sum_{n\ge1}\frac{(-1)^n}{n^s}$ and to finite versions of the zeta function, i.e.  hyper-harmonic numbers of order $r.$ A generalization of these results to the multivariate version of the zeta function $
\zeta\left(p,q\right)=\sum_{k_1>k_2\ge1}\frac{1}{k_1^{p}k_2^{q}},
$ known as multiple zeta values, would constitute an interesting research project. 

Finally, let us finish by circling back to the 
fascinating world of the Borweins. In \cite{BorweinSinc}, a sequence of integrals involving the sinc function 
is shown to exhibit a remarkable consistency pattern which suddenly breaks down with a high precision fraud. These integrals have come to be known as Borwein integrals, and an extension of this result seems an appropriate place to end.
\begin{proposition}
Assume that $\{a_k\}_{0\le k\le n}$ are positive numbers  and that each function $\varphi_{k}$ is the Fourier transform
of an even probability density function (a positive function that integrates to 1) $\psi_{k}$ with bounded support $\left[-\frac{a_k}{2},\frac{a_k}{2}\right]$.\\ 
If $a_0>a_1+\dots +a_n$ then
\[
\int_{0}^{\infty}
\textnormal{sinc}\left(a_{0}z\right) \prod_{k=1}^{n}\varphi_{k}\left(a_{k}z\right)dz
=\frac{\pi}{2a_0},
\]
whereas if $a_0<a_1+\dots +a_n$,
\[
\int_{0}^{\infty}
\textnormal{sinc}\left(a_{0}z\right) \prod_{k=1}^{n}\varphi_{k}\left(a_{k}z\right)dz
<\frac{\pi}{2a_0}.
\]
\end{proposition}
\sloppy Here is a sketch of the proof: 
by the Plancherel Parseval identity, 
this is the integral of the product of the indicator function on the interval 
$\left[-\frac{a_{0}}{2},\frac{a_{0}}{2}\right]$  
with the multiple convolution of the functions $\frac{1}{a_k}\psi\left(\frac{z}{a_k}\right) $, which is a probability density function (pdf) with support 
$\left[-\frac{1}{2}\sum_{k=1}^{n}a_{k},\frac{1}{2}\sum_{k=1}^{n}a_{k}\right].$
In the case $a_0>a_1+\cdots+a_n$, we integrate this pdf over its full support, producing 1. In the case $a_0<a_1+\cdots+a_n$, the pdf is integrated over part of its support, producing a positive number less than $1.$

An example is given by the normalized Bessel function of the first kind, 
\[\varphi_{k}\left(z\right)=j_{\nu_{k}}\left(z\right)=2^{\nu_{k}}\Gamma\left(\nu_{k}+1\right)\frac{J_{\nu_{k}}\left(z\right)}{z^{\nu_{k}}}
\]
with $\nu_{k}\ge-1/2,$
which is the Fourier transform of the symmetric beta pdf 
\[\psi_{k}\left(z\right)=\frac{\Gamma\left(\nu_{k}+1\right)}{\sqrt{\pi}\Gamma\left(\nu_{k}+\frac{1}{2}\right)}\left(1-z^{2}\right)^{\nu_{k}-\frac{1}{2}}\mathrm{1}_{\left[-1,1\right]}\left(z\right),\]
producing, in the case $a_0=1$ and $\nu _k=0,\,1\le k \le n,$ the 
numerical example
\[
\int_{0}^{\infty}\text{sinc}\left(z\right)J_{0}\left(\frac{z}{3}\right)dz = \frac{\pi}{2}
\]
\[
\int_{0}^{\infty}\text{sinc}\left(z\right)J_{0}\left(\frac{z}{3}\right)J_{0}\left(\frac{z}{5}\right)dz = \frac{\pi}{2}
\]
\[
\vdots 
\]
\[
\int_{0}^{\infty}\text{sinc}\left(z\right)J_{0}\left(\frac{z}{3}\right)J_{0}\left(\frac{z}{5}\right)J_{0}\left(\frac{z}{7}\right)J_{0}\left(\frac{z}{9}\right)J_{0}\left(\frac{z}{11}\right)J_{0}\left(\frac{z}{13}\right)dz = \frac{\pi}{2}
\]
\[
\int_{0}^{\infty}\text{sinc}\left(z\right)J_{0}\left(\frac{z}{3}\right)J_{0}\left(\frac{z}{5}\right)
\cdots 
J_{0}\left(\frac{z}{13}\right)J_{0}\left(\frac{z}{15}\right)dz \simeq \left(1-6.267\times 10^{-7}\right)\frac{\pi}{2}.
\]
The previous integrals are notoriously difficult to evaluate numerically due to their oscillatory nature; they are, with the exception of the last one, a special case of the Weber-Schafheitlin
type integrals in the 1979 article \cite[Equation (3.2)]{Srivastava} by Exton
and Strivastava.

Since their discovery, Borwein integrals have been interpreted and extended in multiple ways, including by D.H. and J.M. Borwein themselves, see for example \cite{BorweinMares},
where these integrals are related to volumes of polyhedra. In fact, G. P\'olya makes a reappearance in this context in his 1912 doctoral thesis \cite{PolyaThesis}. See also \cite{BorweinStraub} for further results expressing these integrals as series.

\section*{Acknowledgments}
The authors wish to thank Margarite L. LaBorde and Tanay V. Wakhare for their interesting comments on an early version of this work, as well as Chance Sanford for sending a rare reference related to Gosper's work on infinite products of matrices. The first author acknowledges support from the Department of Defense SMART scholarship program. The second author thanks Tucker and Nivens for their unconditional support.


%


\begin{thebibliography}{10}

\bibitem{Agostini}Agostini, D., Amendola, C. (2019). Discrete Gaussian
Distributions via Theta Functions. {\it SIAM J. Appl. Algebra Geom.}
3(1): 1--30.

\bibitem{BorweinZeta}Bailey, D.H., Borwein, J.M., Bradley, D.M. (2006). Experimental Determination of Apéry-like Identities for $\zeta(2n+2)$. {\it Exp. Math.} 15(3):281--289.

\bibitem{BBDictionary}Borwein, J.M., Borwein, P.B. (1990). A Dictionary of Real Numbers. Boston, MA: Springer. 

\bibitem{BorweinSinc} Borwein, D.H., Borwein, J.M. (2001). Some remarkable properties of sinc and related integrals, {\it Ramanujan J.} 5, 73--90.

\bibitem{Borwein}Borwein, J.M., Borwein, P.B. (1992). Strange Series and High Precision Fraud. {\it Amer. Math. Monthly}. 
99(7):622--640.

\bibitem{BorweinMares}Borwein, D.H., Borwein, J.M., Mares Jr, B.A. (2002).  Multi-variable sinc integrals and volume of polyhedra. {\it Ramanujan J.} 6:189--208.

\bibitem{BorweinStraub}
Borwein, D.H., Borwein, J.M., Straub, A. (2012). A Sinc that Sank, {\it Amer. Math. Monthly}.  119(7):535--549. 



\bibitem{Bradshaw2} Bradshaw, Z.P., LaBorde, M.L., Wilde, M.M. (2023). Cycle index polynomials and
generalized quantum separability tests. {\it Proc. R. Soc. A} 479:20220733.


\bibitem{Finch}Finch, S.R.  (2003). {\it Mathematical Constants}. Encyclopedia of Mathematics and its Applications 94. Cambridge University Press.

\bibitem{Gosper}Gosper, R.W. (1976). Analytic identities from path invariant
matrix multiplication.  Unpublished manuscript.

\bibitem{Gosper2}Gosper, R.W. (1990).
Strip Mining in the Abandoned Orefields of Nineteenth Century Mathematics, available online, https://www.tweedledum.com/rwg/Gosper\, 1990\,Strip\,Mining.pdf





\bibitem{Kemp}Kemp, A.W. (1997).
Characterizations of a discrete normal distribution,
{\it J. Stat. Plan. Inference.} 63(2):223--229,


\bibitem{Kondratieva} Kondratieva, M.,  Sadov,  S. (2005). Markov's transformation of series and the WZ method. {\it Adv. in Appl. Math.} 34(2): 393--407.

\bibitem{LaBorde} LaBorde, M.L., Rethinasamy, S., Wilde, M.M. (2022). Testing symmetry on quantum computers. arxiv:2105.12758.

\bibitem{NIST}
Olver, F.W.J., Lozier,  D.W.,  Boisvert, R.F., Clark, C.W. (2010). {\it The {NIST} Handbook of Mathematical Functions}. 
Cambridge Univ. Press.

\bibitem{PolyaThesis} P\'olya, G. (1912). A val\'oszinus\'egsz\'am\'it\'as n\'eh\'any k\'erd\'es\'erol \'es bizonyos vel\H{uk} \H{o}sszef\H{u}ggo hat\'arozott integr\'alokr\'ol (On a few questions in probability theory and some definite integrals related to them), Ph.D. Thesis, E\H{o}tv\H{o}s Lor\'and University, Budapest, Hungary.

\bibitem{Polya} P\'olya, G. (1937). Kombinatorische Anzahlbestimmungen f\"ur Gruppen, Graphen und ChemischeVerbindungen. {\it Acta Math.} 68:145--254.

\bibitem{Prudnikov}Prudnikov, A.P., Brychkov, Yu. A., Marichev, O.I. (1990). {\it Integrals and Series}. Volume I.
Taylor and Francis.


\bibitem{RomikTheta}Romik, D. (2020).  The Taylor coefficients of the Jacobi theta constant $\theta_3$. {\it Ramanujan J.} 52:275--290


\bibitem{Srivastava}Srivastava, H.M., Exton, H. (1979). A generalization
of the Weber-Schafheitlin integral, 
{\it J. f\"{u}r die Reine und Angew. Math.}
309, 1--6.






\bibitem{Wakhare 3}Wakhare, T., Vignat, C. (2023). Infinite matrix products and hypergeometric zeta series. arXiv:2301.00298.

\bibitem{Romik}Wakhare, T., Vignat, C. (2020). Taylor coefficients of the Jacobi $\theta_3(q)$ function. {\it J. Number Theory}. 216:280--306. 


\end{thebibliography}
\end{document}